\documentclass{amsart}
\usepackage{amssymb}
\usepackage{graphicx}

\newtheorem{theorem}{Theorem}[section]

\newtheorem{problem}[theorem]{Problem}

\theoremstyle{definition}

\newtheorem{corollary}[theorem]{Corollary}

\theoremstyle{remark}
\newtheorem{remark}[theorem]{Remark}

\numberwithin{equation}{section}

\DeclareMathOperator{\tr}{tr}
\DeclareMathOperator{\diag}{diag}

\title{Leading Coefficients and the Multiplicity of Known Roots}

\author{Gregory J. Clark}
\address{Department of Mathematics, University of South Carolina, Columbia, South Carolina 29208}
\email{gjclark@math.sc.edu}
\thanks{This work was partially supported by a SPARC Graduate Research Grant
from the Office of the Vice President for Research at the University of South Carolina}

\author{Joshua N. Cooper}
\address{Department of Mathematics, University of South Carolina, Columbia, South Carolina 29208}
\email{cooper@math.sc.edu}


\date{June 11, 2018}


\begin{document}
\maketitle
\begin{abstract}
We show that a monic univariate polynomial over a field of characteristic zero, with $k$ distinct non-zero known roots, is determined by its $k$ proper leading coefficients by providing an explicit algorithm for computing the multiplicities of each root.  We provide a version of the result and accompanying algorithm when the field is not algebraically closed by considering the minimal polynomials of the roots.  Furthermore, we show how to perform the aforementioned algorithm in a numerically stable manner over $\mathbb{C}$, and then apply it to obtain new characteristic polynomials of hypergraphs.
\end{abstract}

\section{Introduction} \label{sec:intro}
In \cite{Gyo}, Gy\H ory et al. present the following natural question.
\begin{problem}
\label{P:1}
Let $K$ be a field of characteristic zero. Is it true that a monic polynomial $p = \sum_{i=0}^n c_ix^{n-i}\in K[x]$ of degree $n$ with exactly $k$ distinct zeros is determined up to finitely many possibilities by any $k$ of its non-zero coefficients?
\end{problem}

By degrees-of-freedom considerations, at least $k$ coefficients are needed; which sets of $k$ coefficients actually suffice, however, seems to be a delicate matter.  We consider the following variation of Problem \ref{P:1}.  The {\em codegree} of a monomial term $Cx^k$ in a univariate polynomial $f(x)$ is $\deg(f)-k$.  A set of coefficients is {\em leading} if it corresponds to codegrees $\{0,\ldots,k\}$ for some $k$; it is {\em proper leading} if it corresponds to codegrees $\{1,\ldots,k\}$ for some $k$.

\begin{problem}
\label{P:2}
Let $K$ be a field of characteristic zero.  Is it true that a monic polynomial $p = \sum_{i=0}^n c_ix^{n-i}\in K[x]$ of unknown degree $n$, with exactly $k$ distinct known zeros $r_1, r_2, \dots, r_k$, is uniquely determined by its first $k$ proper leading coefficients?
\end{problem}

We answer Problem \ref{P:2} in the affirmative with the following result.

\begin{theorem}
\label{T:1}
Let $p = \sum_{i=0}^n c_ix^{n-i} \in K[x]$ be a monic polynomial with $k+1$ distinct roots, $r_0 = 0, r_1, r_2, \dots, r_k$, with multiplicity $m_0, m_1, m_2, \dots m_k$, respectively.  Then the multiplicities are uniquely determined by $c_0=1,c_1, \dots, c_k$.
\end{theorem}

Furthermore, $p$ may be determined by fewer than $k$ proper coefficients when $K$ is not algebraically closed. 

\begin{theorem}
\label{T:4}
Let $p = \sum_{i=0}^n c_ix^{n-i} \in K[x]$ be a monic polynomial such that $p(0) \neq 0$.  Suppose $p = \prod_{i=1}^t q_i^{m_i}$ for $q_i \in K[x]$.  The multiplicity vector ${\bf m} = \langle m_1, \dots, m_t \rangle^T$ is uniquely determined by the first $t$ proper coefficients if and only if $V \in \overline{K}^{t \times t}$ is non-singular where 
\[
V_{i,j} = \sum_{r : q_i(r) = 0} r^j.
\]
\end{theorem}

\begin{remark}
 Observe that when $q_i = x-r_i$ (i.e., $p$ splits over $K$) Theorem \ref{T:4} provides the same conclusion as Theorem \ref{T:1}.
\end{remark}

In Section 2 we prove both of the main results.  In particular, we prove Theorem \ref{T:1} via an algorithm which allows us to compute exactly the multiplicity of each root.  In Section 3 we prove that this algorithm is numerically stable in the sense that the requisite number of bits of precision to approximate each root in order to compute its multiplicity {\em exactly} is linear in $k$ and the logarithms of (a) the ratio between the largest and smallest difference of roots, (b) the largest root, and (c) the largest coefficient of codegree at most $k$.  We conclude by demonstrating the utility of this algorithm by computing  previously unknown characteristic polynomials of two 3-uniform hypergraphs in Section 4.

\section{Proof of Main Results}

We begin with a proof of Theorem \ref{T:1}.

\begin{proof}[Proof of Theorem \ref{T:1}]
 Fix such a monic polynomial $p$ with distinct roots $r_0 = 0$, $r_1$, $\ldots$, $r_k$ with respective multiplicities $m_0$, $m_1$, $\ldots$, $m_k$.  Ignoring $r_0$ for a moment, let ${\bf r} = \langle r_1, \dots, r_k\rangle^T$ and ${\bf m} = \langle m_1, \dots, m_k\rangle^T$.   We denote the  Vandermonde matrix 
 \[ 
 V = \left( \begin{array}{cccc}
1 & 1 & \dots & 1 \\
r_1 & r_2 & \dots & r_k \\
r_1^2 & r_2^2 & \dots & r_k^2 \\
\vdots & \vdots & \vdots & \vdots \\
 r_1^{k-1} & r_2^{k-1}& \dots & r_k
^{k-1}\end{array} \right)
 \] 
 and consider 
 \[
 V_0 = \left( \begin{array}{cccc}
r_1 & r_2 & \dots & r_k \\
r_1^2 & r_2^2 & \dots & r_k^2 \\
\vdots & \vdots & \vdots & \vdots \\
 r_1^{k} & r_2^{k}& \dots & r_n^{k}\end{array} \right).
 \]  
 Let ${\bf p} \in \overline{K}^n$ where
 \[
 {\bf p}_i := \sum_{j=1}^k r_j^i m_j
 \]
 then 
 \[
 V_0{\bf m} = {\bf p}.
 \]
  Notice that $V_0 =  V \diag({\bf r})$ is non-singular as it is the product of two non-singular matrices.  We have then
 \begin{equation}
 \label{E:m}
 {\bf m} = V_0^{-1}{\bf p}.
 \end{equation}
 We present a formula for ${\bf p}$ which is a function of only the leading $k+1$ coefficients of $p$. 
 
 Let $A$ be the diagonal matrix where $r_i$ occurs $m_i$ times and note
 \[
 p(x) = \det(xI - A).
 \]
 By the Faddeev-LeVerrier algorithm (aka the Method of Faddeev, \cite{Gan}) we have for $j \geq 1$
 \begin{equation}
 \label{E:Faddeev}
 c_{j} = -\frac{1}{j}\sum_{i=1}^j c_{j-i} \tr(A^i) = -\frac{1}{j}\sum_{i=1}^j c_{j-i} {\bf p}_i.
 \end{equation}
Let ${\bf c} = \langle c_1, c_2, \dots, c_{k} \rangle^T$, $\Lambda  = -\diag(1, 1/2, \dots, 1/k)$, and 
\[ 
C = \left( \begin{array}{cccc}
c_0=1 & 0 & \dots & 0 \\
c_1 & c_0 & \dots & 0 \\
\vdots & \vdots & \ddots & \vdots \\
 c_{k-1} & c_{k-2}& \dots & c_0\end{array} \right).
\]
By Equation \ref{E:Faddeev},
${\bf c} = \Lambda C{\bf p}$.  Moreover, as $\Lambda $ and $C$ are invertible we have ${\bf p} = (\Lambda C)^{-1}{\bf c}$.  It follows from Equation \ref{E:m} that
\begin{equation}
\label{E:formula}
{\bf m} = (\Lambda CV_0)^{-1}{\bf c}.
\end{equation}
Furthermore, $m_0 = n - {\bf 1}\cdot{\bf m}$.
\end{proof}

We briefly remark about the proof of Theorem \ref{T:1}.  Problem \ref{P:1} has a flavor of polynomial interpolation: given $k$ points, how many (univariate) polynomials of degree $n$ go through each of the $k$ points?  If $n \leq k-1$ the polynomial is known to be unique and is relatively expensive to compute (as any standard text in numerical analysis will attest).  Our proof technique mimics this approach as the classical problem of determining $p = \sum_{i=0}^{k-1}{c_i}x^{(k-1)-i}$, which resembles $k$ distinct points $\{(x_i,y_i)\}_{i=1}^k$, can be solved by computing
\[
V^T{\bf c} = {\bf y}
\]
where ${\bf c} = \langle c_{k-1}, c_{k-2}, \dots, c_0\rangle^T$, $V$ is as previously defined given $r_i = x_i$, and ${\bf y} = \langle y_1, \dots, y_k\rangle^T$.  Suppose for a moment that each root is distinct so that
\[
p(x) = \prod_{i=1}^k (x - r_i).
\]
Then $c_j$, the codegree $j$ coefficient, is precisely the $j$th elementary symmetric polynomial in the variables $r_1, \dots, r_k$.  In the case of repeated roots we have that $c_j$ can be expressed using modified symmetric polynomials in the distinct roots where $r_i$ is replaced with $\binom{m_i}{j}r_i^j$.  The expression for each coefficient via these modified symmetric polynomials is given by Equation \ref{E:Faddeev}.  

Note that if the roots of $p$ are known, it is possible to determine $p$ with fewer coefficients than the number of distinct roots (e.g., when $p$ is a non-linear minimal polynomial).  We modify Theorem \ref{T:1} to include the case when some of the roots are known to occur with the same multiplicity.  We now prove Theorem \ref{T:4}.

\begin{proof}[Proof of Theorem \ref{T:4}]
The proof follows similarly to that of Theorem \ref{T:1}.  First suppose that $V$ is non-singular. Let ${\bf m} = \langle m_1, \dots, m_t \rangle^T$ and 
\[
{\bf p}_i = \sum_{j=1}^t \left(\sum_{r : q_j(r) = 0} r^i\right)m_j = (V{\bf m})_i
\]
so that if $V$ is non-singular, ${\bf m} = V^{-1}{\bf p}$.  Let $A$ be defined as in Theorem \ref{T:1}: the diagonal matrix where the roots of $q_i$ occur $m_i$ times and note
 \[
 p(x) = \det(xI - A).
 \]
We have by the the Faddeev-LeVerrier algorithm, for $j \geq 1$
\[
 c_{j} = -\frac{1}{j}\sum_{i=1}^j c_{j-i} \tr(A^i) = -\frac{1}{j}\sum_{i=1}^j c_{j-i} {\bf p}_i
\]
so that for ${\bf c} = \langle c_1, \dots, c_t\rangle^T$, $\Lambda  = -\diag(1, 1/2, \dots, 1/t)$ we have
\[
{\bf m} = (\Lambda CV)^{-1}{\bf c}.
\]
If instead ${\bf m}$ is uniquely determined by the first $t$ proper coefficients then $V{\bf m} = (\Lambda C)^{-1}{\bf c}$ has exactly one solution, hence $V$ is non-singular.
\end{proof}

As a non-example, consider the minimal polynomial of $\alpha = \sqrt{2}$, $q_{\alpha}(x) = x^2-2 \in \mathbb{Z}[x]$ and suppose 
\[
p = q_\alpha^d  = x^{2d} + 0x^{2d-1} -2dx^{2d-2} + \dots\in \mathbb{Z}[x].
\]
Observe that we cannot determine $d$ given $c_1 = 0$; moreover, this conclusion is unsurprising, given the hypotheses from Theorem \ref{T:4}, since
\[
V = [\sqrt{2} - \sqrt{2}] = [0] \in \mathbb{C}^{1,1}
\]
is singular.  However, by inspection we could determine $d$ given $c_2 = -2d$ and in fact the matrix $[\sqrt{2}^2 + (-\sqrt{2})^2] = [4] \in \mathbb{C}^{1 \times 1}$ is non-singular.  Indeed we could determine $d$ with a simple change of variable: apply Theorem \ref{T:4} to $p_0 = (y-2)^d$ where $y = x^2$.

\section{The stability of computing multiplicities}

We now consider the feasibility of computing ${\bf m}$.  In general, the matrix $V_0$ in Theorem \ref{T:1} may be poorly conditioned, so this calculation is often difficult to carry out even for modest values of $k$.
The goal of this section is to show that if each root of a monic polynomial $p(x) \in \mathbb{Z}[x]$ is approximated by a disk of radius at most $\epsilon$, a ``reasonable'' precision, then the interval approximating ${\bf m}_i$, resulting from a particular algorithm, contains exactly one integer.  That is, we provide an algorithm for exactly computing ${\bf m}$ via \cite{Sag} with substantially improved numerical stability over simply following the calculations in Section \ref{sec:intro}.

\begin{theorem}
\label{T:2}
Let $p(x) = \sum_{i=0}^t c_ix^{t-i} \in \mathbb{Z}[x]$ be a monic polynomial with distinct non-zero roots $r_1,\dots, r_n$ such that $|r_1| \geq |r_2| \geq \dots \geq |r_n| > 0$.  If each root is approximated by a disk of radius $\epsilon$ such that
\[
\epsilon < \frac{m^2r}{2^{2n+7}n^5} \left(\frac{m}{MRc}\right)^n = \left ( \frac{m}{MRc} \right)^{n(1+o(1))}
\]
where
\begin{itemize}
    \item{$M = \max\{\max\{|r_i - r_j|\}, 1\}$ and $m = \min\{\min\{|r_i - r_j|: i \neq j\}, 1\}$}
    \item{$R = \max\{|r_1|,1\}$ and $r = \min\{|r_n|,1\}$}
    \item{$c = \max\{\max\{|c_i|\}_{i=1}^n,1\}$.}
\end{itemize} 
then the resulting disk approximating ${\bf m}_i = (\Lambda CV_0)^{-1}{\bf c}_i \in \mathbb{Z}$ contains exactly one integer (i.e., the computation of ${\bf m}$ is stable). 
\end{theorem}

Notice that $MRc \geq 1$ and $MRc  = 2$ for $x^n -1$ when $n$ is even.  Roots of unity occur frequently in the spectrum of hypergraphs; see Section 3.  In particular, $k$-cylinders -- essentially $k$-colorable $k$-graphs -- have a spectrum which is invariant under multiplication by the $k$th roots of unity.  Consider now $p(x) = x^n-1$.  We have $m = \sqrt{2 - 2\cos(\frac{2\pi}{n})}$ so that
\[
\epsilon < \frac{(2 - 2 \cos(\frac{2\pi}{n}))^\frac{n+2}{2}}{2^{3n+7}n^5}.
\]
While $\epsilon$ may seem small, we are chiefly concerned with the number of bits of precision needed to approximate each root.  Indeed for $x^n-1$ we need 
$|\ln \epsilon| = O(n \ln n)$ bits of precision by the small-angle approximation.  

\begin{remark}
 The bound on $\epsilon$ is ``reasonable'', as the number of bits required to approximate each root is proportional to the number of distinct roots of $p$ and the logarithms of the ratio of the smallest difference of the roots with the largest difference of roots, the largest root, and the largest coefficient.
\end{remark}

In practice, the difficulty of computing ${\bf m}$ as described in Theorem \ref{T:1} is in computing the inverse of the Vandermonde matrix, whose entries may vary widely in magnitude and which may be very poorly conditioned.  The task of inverting Vandermonde matrices has been studied extensively.  In \cite{Eis}, Eisenberg and Fedele provide a brief history of the topic as well results concerning the accuracy and effectiveness of several known algorithms.  However, these algorithms provide good approximations for the entries of $V^{-1}$, whereas we seek to express them exactly as elements of the field of algebraic complex numbers, since ${\bf m}$ is a vector of integers.  In \cite{Sot}, Soto-Eguibar and Moya-Cessa showed that $V^{-1} = \Delta WL$ where $\Delta $ is the diagonal matrix
 \begin{displaymath}
   \Delta_{i,j} = \left\{
     \begin{array}{ll}
       \prod_{k=1, k \neq i}^n \frac{1}{r_i - r_k}& : i=j \\
       0 & : i \neq j,
     \end{array}
   \right.
\end{displaymath} 
$W$ is the lower triangular matrix
 \begin{displaymath}
   W_{i,j} = \left\{
     \begin{array}{ll}
       0 & : i> j \\
       \prod_{k=j+1, k \neq i}^n (r_i - r_k) & : \text{ otherwise},
     \end{array}
   \right.
\end{displaymath} 
and $L$ is the upper triangular matrix
 \begin{displaymath}
   L_{i,j} = \left\{
     \begin{array}{ll}
       0 & : i< j \\
       1 & : i = j\\
       L_{i-1, j-1} - L_{i-1, j-1} r_{i-1} & : i \in [2,n], j \in [2,i-1].
     \end{array}
   \right.
\end{displaymath}
Using this decomposition, it is possible to compute ${\bf m}$ exactly. To prove Theorem \ref{T:2} we first provide an upper bound for the diameter of the disk approximating an entry of $\Delta $, $W$, and $L$, respectively; to do so, we extensively employ computations of \cite{Pet} found in Chapter 1.3.  We present the necessary background here.  

Let $D(z, \epsilon)$ be the open disk in the complex plane centered at $z$ of radius $\epsilon$.  For $A = D(a,r_1)$, $B= D(b,r_2)$ complex open disks, we have
\begin{enumerate}
\item {$A \pm B = D(a\pm b, r_1 + r_2)$}
\item{$1/B = D\left(\frac{\bar{b}}{|b|^2 - r_2^2}, \frac{r_2}{|b|^2 - r_2^2}\right)$}
\item{$AB = D(ab, |a|r_2 + |b|r_2 + r_1r_2)$}
\end{enumerate}
In particular, for the special case of $A^n$ we have
\begin{equation}
\label{E:power}
D(a,r_1)^n = D(a^n, (|a| - r_1)^n - |a|^n).
\end{equation}
Moreover, given $0 < r_1 < 1 \leq |a|$
\begin{equation}
\label{E:bound}
 (|a| - r_1)^n - |a|^n \leq r_1(2|a|)^n
\end{equation}
since
\[
(|a| - r_1)^n - |a|^n \leq \sum_{k=1}^n \binom{n}{k} r_1^k |a|^{n-k}
\leq r_1(2|a|)^n.
\]
Finally, let $d(A) = 2r_1$ denote the diameter of $A$ and let 
\[
|A|= |a| + r_1
\] 
be the absolute value of $A$.  Then for $u \in \mathbb{C}$ we have
\begin{enumerate}
\item{$d(A\pm B) = d(A) + d(B)$}
\item{$d(uA) = |u| d(A)$}
\item{$d(AB) \leq |B|d(A) + |A|d(B)$}
\end{enumerate}

For the remainder of this paper some numbers will be exact (e.g., rational numbers) while others will be approximated by a disk.  The non-exact entries of a matrix $M \in \mathbb{C}^{n \times n}$ will be referred to as disks; this will be clear from the problem formulation or derived from the computations.  With a slight abuse of notation we use $d(M_{i,j})$ and $|M_{i,j}|$ to denote the diameter and absolute value of the disk approximating the entry $M_{i,j}$.  Moreover, we write
\[
d(M) = \max\{d(M_{i,j}) : i,j \in [n]\} \text{ and } |M| = \max\{|M_{i,j}| : i,j \in [n]\}.
\]
In the case when the entry is exact, the diameter is zero and the absolute value (of the disk) is simply the modulus.  

\begin{theorem}
\label{T:3}
Assume the notations of Theorem \ref{T:4}, let $V$ denote the Vandermonde matrix from the proof of Theorem \ref{T:1}, and let $V^{-1} = \Delta WL$ by \cite{Sot}.  Then
\[
d(V^{-1}) \leq  \frac{2^{2n+4}n}{m^2}\left(\frac{MR}{m}\right)^n\epsilon.
\]
and 
\[
|V^{-1}|\leq  2n\left(\frac{RM}{m}\right)^n.
\]
\end{theorem}

 \begin{proof}
 Let 
\[
D_i := D(r_i, \epsilon)
\]
denote the disk centered at $r_i$ with radius $\epsilon$.  By Equation \ref{E:bound} we have for $s \neq t$
 \begin{align*}
 d(\Delta)&\leq d \left ( \left(\frac{1}{D_s - D_t}\right)^n \right ) \\
 &\leq 2^n\left(\frac{2\epsilon}{m^2-(2\epsilon)^2}\right)\left(\frac{m}{m^2-(2\epsilon)^2}\right)^n \\
 &\leq \frac{2^{2n+2}}{m^{n+2}} \cdot \epsilon ,
 \end{align*}
 since $\epsilon < m/4$, 
 \[
 d(W) \leq d((D_s - D_t)^n) \leq 2^{n+1}M^n\epsilon,
 \]
 and 
 \[
 d(L) \leq d(D_s^n) \leq (2R)^n \epsilon.
 \]
   We first consider $d(\Delta W)$.  Observe that $\Delta W$ is upper triangular and each non-zero entry of $\Delta W$ is a product of exactly one non-zero entry of $\Delta $ and $W$.  In this way
 \[
 d((\Delta W)_{i,j}) \leq |W_{i,j}|d(\Delta _{i,i}) + |\Delta _{i,i}|d(W_{i,j})
 \leq\frac{2^{2n+3}}{m^2} \left(\frac{M}{m}\right)^n\epsilon
 \]
and
 \[
    |\Delta W| \leq 2\left(\frac{M}{m}\right)^n.
\]
 
 We now determine $d(\Delta WL)$ by first computing
 \[
 d((\Delta W)_{i,k} L_{k,j}) \leq |L_{k,j}|d(\Delta W_{i,k}) + |\Delta W_{i,k}|d(L_{k,j}) 
 \leq  \frac{2^{2n+4}}{m^2}\left(\frac{RM}{m}\right)^n\epsilon.
 \]
 Hence 
 \[
 d(V^{-1}) = d(\Delta WL) \leq \max_{i,j} \sum_{k=1}^n d((\Delta W)_{i,k}L_{k,j}) \leq\frac{ 2^{2n+4}n}{m^2}\left(\frac{RM}{m}\right)^n\epsilon
 \]
and
 \[
     |V^{-1}| \leq 2n\left(\frac{RM}{m}\right)^n.
 \]
 \end{proof}
 
  In our computations we are concerned with $V_0 = V \cdot \diag({\bf r})$ where $\diag({\bf r}) = \diag(r_1, \dots, r_n)$ so that
 \[
 V_0^{-1} = \diag({\bf r})^{-1}V^{-1}.
 \]
 The following Corollary is immediate from the observation that 
 \[
 d(\diag({\bf r})^{-1}) \leq \frac{2}{r}.
 \]
\begin{corollary}
\[
d(V^{-1}_0)\leq   \frac{2^{2n+6}n}{m^2r}\left(\frac{MR}{m}\right)^n\epsilon
\]
and
\[
|V_0^{-1}| \leq \frac{2n}{r} \left(\frac{RM}{m}\right)^n.
\]
\end{corollary}

 We are now able to prove Theorem \ref{T:2}.

\begin{proof}[Proof of Theorem \ref{T:2}]
Recall ${\bf m} = V_0^{-1} C^{-1}\Lambda ^{-1}{\bf c}$ as defined in the proof of Theorem \ref{T:1}.  Fortunately, the remainder of the computations are straightforward as $C^{-1}, \Lambda ^{-1}$, and ${\bf c}$ have integer, and thus exact, entries.  As 
 \begin{displaymath}
   C_{i,j}^{-1} = \left\{
     \begin{array}{lr}
       0 & : i < j \\
       1 & : i = j\\
       -\sum_{k=1}^{i-1} c_{i-k} C_{k,j}^{-1}  & : i > j
     \end{array}
   \right.
\end{displaymath}
we have
\[
d(V_0^{-1}C^{-1}) \leq n (nc^{n-1}) \frac{2^{2n+6}n}{m^2r}\left(\frac{MR}{m}\right)^n\epsilon =
\frac{2^{2n+6}n^3}{m^2rc}\left(\frac{MRc}{m}\right)^n\epsilon
\]
Further, since $\Lambda ^{-1} = -\diag(1,2,\dots,n)$ we have
\[
d(V_0^{-1}C^{-1}\Lambda ^{-1}) \leq |-n| d(V_0^{-1}C^{-1}) = \frac{2^{2n+6}n^4}{m^2rc}\left(\frac{MRc}{m}\right)^n\epsilon
\]
and, finally,
\begin{align*}
d(V_0^{-1}C^{-1}\Lambda ^{-1}){\bf c}) &\leq nc \cdot d(V_0^{-1}C^{-1}\Lambda ^{-1}) \\
&\leq \frac{2^{2n+6}n^5}{m^2r}\left(\frac{MRc}{m}\right)^n\epsilon < \frac{1}{2}.
\end{align*}
Thus each interval will contain at most one integer as desired. 
 
\end{proof}

\section{Application to Hypergraph Spectra}
 For the present authors, Problem \ref{P:2} arose organically in the context of spectral hypergraph theory.  In short, the authors were concerned with determining high-degree polynomials when the roots (without multiplicity) are known and all but the lowest-codegree coefficients are too costly to compute.  We briefly explain the context of spectral hypergraph theory for those interested in the origin of such questions.  However, our presentation of the computations is self-contained: the reader who wishes to see Theorem \ref{T:1} applied immediately may skip the next few paragraphs. 
 
For $k \geq 2$, a {\em $k$-uniform hypergraph} is a pair ${\mathcal H} = (V, E)$ where $V = [n]$ is the set of \em{vertices} and $E \subseteq \binom{[n]}{k}$ is the set of \em{edges}.  It is common to refer to such hypergraphs as {\em $k$-graphs} when $k > 2$ and as just {\em graphs} when $k =2$.  We are particularly interested in the computation of the characteristic polynomial of a uniform hypergraph.  The characteristic polynomial of the adjacency matrix of a graph is straightforward to compute; however, the same cannot be said for hypergraphs.  The {\em characteristic polynomial} of the {\em (normalized) adjacency hypermatrix} $\mathcal{A}$ of ${\mathcal H}$, denoted $\phi_{{\mathcal H}}(\lambda)$, is the resultant of a family of $|E|$ homogeneous polynomials\footnote{Namely, the Lagrangians of the links of all vertices minus $\lambda$ times the $(k-1)$-st power of the corresponding vertices' variables, or, equivalently, the coordinates $f_v$ of the gradient of the $k$-form naturally associated with $\mathcal{A}-\lambda \mathcal{I}$.  The (symmetric) hyperdeterminant is the unique irreducible monic polynomial in the entries of $\mathcal{A}$ whose vanishing corresponds exactly to the existence of nontrivial solutions to the system $\{f_v = 0\}_{v \in V(\mathcal{H})}$.} of degree $k-1$ in the indeterminate $\lambda$; the {\em order $k$}, {\em dimension $n$} hypermatrix $\mathcal{A} \in \mathbb{R}^{[n]^k}$, whose rows and columns are indexed by the vertices of $\mathcal{H}$ and whose $(v_1,\ldots,v_k)$ entry is $1/(k-1)!$ times the indicator of the event that $\{v_1,\ldots,v_k\}$ is an edge of $\mathcal{H}$ is also sometimes called the {\em adjacency tensor} of $\mathcal{H}$. Equivalently, one can define the characteristic polynomial to be the hyperdeterminant of $\mathcal{A}-\lambda \mathcal{I}$ (as in \cite{Gel}), where $\mathcal{I}$ is the identity hypermatrix, i.e., $\mathcal{I}(v_1,\ldots,v_k)$ is the indicator of the event that $v_1 = \cdots = v_k$.  The set $\sigma({\mathcal H}) = \{r : \phi_{{\mathcal H}}(r) = 0\} \subset \mathbb{C}$ is the {\em spectrum} of ${\mathcal H}$ and each $r \in \sigma({\mathcal H})$ is an {\em eigenvalue} of ${\mathcal H}$.  It is known that $\phi_{\mathcal H}(\lambda)$ is a monic polynomial of degree $n(k-1)^{n-1}$, and many of the properties of characteristic polynomials of graphs generalize nicely to hypergraphs; we refer the interested reader to \cite{Coo} and \cite{Qi} for further exploration of the topic. 

Given a $k$-graph ${\mathcal H}$ we aim to compute $\phi_{\mathcal H}(\lambda)$.  Unfortunately, the resultant is known to be NP-hard to compute (\cite{Gre}) despite its utility in several fields of mathematics, perhaps nowhere more so than computational algebraic geometry.  Nonetheless, one can attempt to imitate classical approaches to computing characteristic polynomials of ordinary graphs.  In particular, Harary \cite{Har} (and Sachs \cite{Sac}) showed that the coefficients of $\phi_G(\lambda)$ can be expressed as a certain weighted sum of the counts of  subgraphs of $G$.  The authors have established an analogous result for the coefficients of $\phi_{{\mathcal H}}(\lambda)$ \cite{Cla2}.  This formula allows one to compute many low codegree coefficients -- i.e., the coefficients of $x^{d-k}$ for $k$ small and $d = \deg(\phi_{\mathcal{H}})$ -- by a certain linear combination of subgraph counts in $\mathcal{H}$.  Unfortunately, this computation becomes exponentially harder as the codegree increases, making computation of the entire (often extremely high degree) characteristic polynomial impossible for all but the simplest cases.  A method of Lu-Man \cite{Lu}, ``$\alpha$-normal labelings'' is an alternative approach that can obtain all eigenvalues with relative efficiency, but it gives no information about their multiplicities.  {\it Combining} these two techniques, however, yields a method to obtain the full characteristic polynomial: obtain a list of roots, compute a few low-codegree coefficients using subgraph counts, and then deduce the roots' multiplicities.  Therefore, we arrive at the following special case of Problem \ref{P:2}.

\begin{problem}
\label{P:3}
Let $K$ be a field of characteristic zero.  Is it true that a monic polynomial $p \in K[x]$ of degree $n$ with exactly $k$ distinct, known roots is determined by its $k$ proper leading coefficients?
\end{problem}

Returning to our application of Theorem \ref{T:1}, we can compute $\phi_{\mathcal H}(\lambda)$ if we know $\sigma({\mathcal H})$ and the first $|\sigma({\mathcal H})|$ coefficients (note this includes coefficients which are zero as well as the leading term).  In \cite{Lu}, Lu and Man introduced {\em $\alpha$-consistent incidence matrices} which can be used to find the eigenvalues of ${\mathcal H}$ whose corresponding eigenvector has all non-zero entries.  These eigenvalues are referred to as {\em totally non-zero eigenvalues} and we denote the set of totally non-zero eigenvalues of ${\mathcal H}$ as $\sigma^+({\mathcal H})$.  The authors showed in \cite{Cla} that for $k > 2,$
\[
\sigma({\mathcal H}) \subseteq \bigcup_{H \subseteq {\mathcal H}} \sigma^+(H)
\]
where $H = (V_0,E_0) \subseteq {\mathcal H}$ if $V_0 \subseteq V$ and $E_0 \subseteq E$ (c.f.~Cauchy Interlacing Theorem when $k=2$).  Computing $\sigma({\mathcal H})$ by way of $\sigma^+(H)$ involves solving smaller multi-linear systems than the one involved in computing $\phi_{\mathcal H}(\lambda)$.  Generally speaking, $|\sigma({\mathcal H})|$  is considerably smaller than the degree of $\phi_{\mathcal H}(\lambda)$.  In practice, this approach has yielded $\phi_{\mathcal H}(\lambda)$ when other approaches of computing $\phi_{\mathcal H}(\lambda)$ via the resultant have failed. 

We present two examples demonstrating these computations.  Consider the \emph{hummingbird hypergraph} ${\mathcal B} = ([13],E)$ where
\[
E = \{ \{1,2,3\},\{1,4,5\},\{1,6,7\},\{2,8,9\},\{3,10,11\},\{3,12,13\}\}.
\]
We present a drawing of ${\mathcal B}$ in Figure \ref{F:1} where the edges are drawn as shaded in triangles.  Note that 
\[
\deg(\phi_{\mathcal B}) = n(k-1)^{n-1} = 13\cdot2^{12} = 53248
\]
and, since ${\mathcal B}$ is a hypertree (and thus a 3-cylinder), its spectrum is invariant under multiplication by any third root of unity \cite{Coo}.  We compute the minimal polynomials of the totally non-zero eigenvalues of $\phi_{\mathcal B}$ via \cite{Mac},
\begin{align*}
    \phi_{\mathcal B} = & x^{m_0}(x^9-6x^3 + 8x^3-4)^{m_1}(x^9-5x^6+5x^3-2)^{m_2} \\ &\cdot (x^3-1)^{m_3}(x^6-4x^3+2)^{m_4}(x^9-4x^6+3x^3-1)^{m_5}\\
    &\cdot (x^6-3x^3+1)^{m_6}(x^3-3)^{54}(x^3-2)^{m_7}.
\end{align*}
With the intent of applying Theorem \ref{T:1} to $\phi_{\mathcal B}$ we consider the change of variable $y = x^3$ and observe that we need to determine $c_3, c_6, \dots, c_{48}$ as there are sixteen distinct nonzero cube roots.  We compute
\begin{align*}
c_3 &= -18432 \\
c_6 &= 169843968 \\ 
c_9 &= -1043209971456 \\ 
c_{12} &= 4804960103034624 \\
c_{15} &=-17702435302276375440 \\
c_{18} &= 
54341319772238850901668 \\ c_{21} &= -142960393819753656566577552 \\
c_{24} &= 329036832924106136747171871042 \\  c_{27} &= -673063350744784559041302787109576 \\ c_{30}&= 1238925078774563882036470496247467682 \\ c_{33}&= -2072891735870949695930286542580991559916 \\ c_{36} &= 3178738418917825954994865036362341584776658 \\ c_{39} &= -4498896549573513724009044022281523093964642496 \\ c_{42} &= 5911636016042739623328802656744094043553245557890 \\ c_{45} &= -7249053168113446546908444934275696322928768819713512 \\ c_{48} &=  8332230937213678426258491158832963153453272812465162851
\end{align*}
Using Theorem \ref{T:2} we have $M < 3, m > .14, R < 4.39,r > .38,$  and $c= c_{48}$ so that each root of $\phi_{\mathcal B}$ needs to be approximated to at most 3091 bits of precision. Using SageMath (\cite{Sag}), we obtain
\begin{align*}
    \phi({\mathcal B}) =& x^{20983}(x^9-6x^3 + 8x^3-4)^{729}(x^9-5x^6+5x^3-2)^{972} \\ &\cdot (x^3-1)^{1782}(x^6-4x^3+2)^{486}(x^9-4x^6+3x^3-1)^{324}\\
    &\cdot (x^6-3x^3+1)^{216}(x^3-3)^{54}(x^3-2)^{119}.
\end{align*}
In Figure \ref{F:1} we provide a plot of $\sigma(\phi_{\mathcal B})$ drawn in the complex plane where a disk is centered at each root and each disk's area is proportional to the algebraic multiplicity of the underlying root in $\phi_{\mathcal B}$.

\begin{figure}[ht]
\begin{center}
\includegraphics[scale = .48]{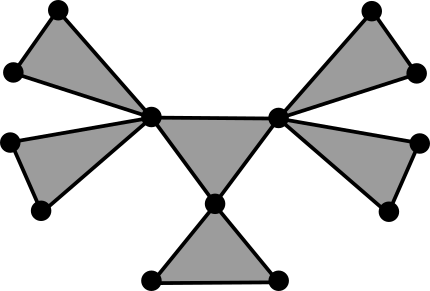} 
\hspace{1cm}
\includegraphics[scale = .25]{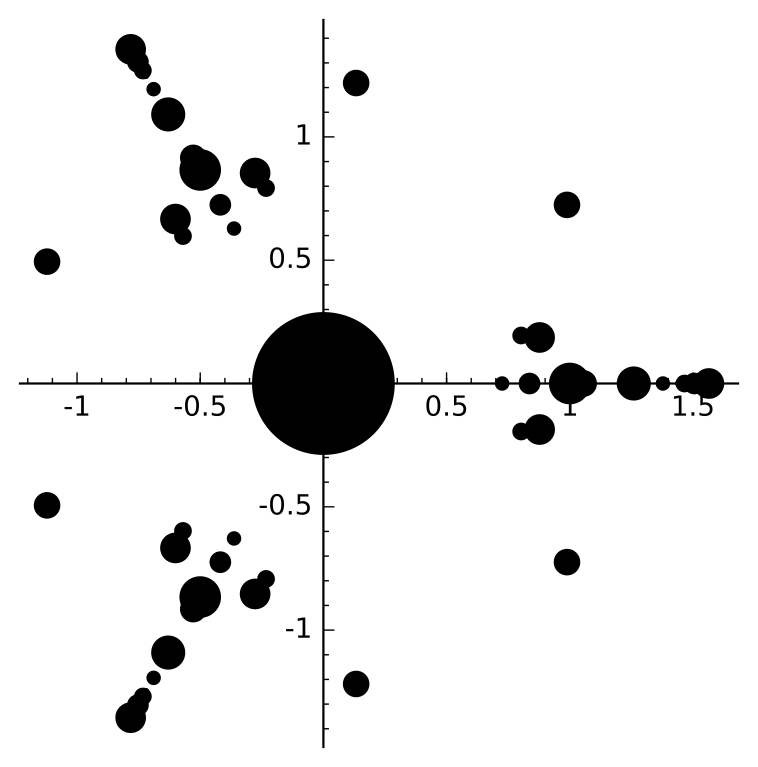} 
\end{center}
\caption{The hummingbird hypergraph and its spectrum.}
\label{F:1}
\end{figure}

Now consider the \emph{Rowling hypergraph}\footnote{The name was chosen for its resemblance to an important narrative device in \cite{Row}.}
\[
{\mathcal R} = ([7], \{1,2,3\},\{\{1,4,5\},\{1,6,7\},\{2,5,6\},\{3,5,7\}\}).
\]
A drawing of ${\mathcal R}$ is given in Figure \ref{F:2} where the edges are drawn as arcs and its spectrum is drawn similarly to that of $\phi_{{\mathcal B}}$; note that ${\mathcal R}$ is also the Fano plane minus two edges.  We have
\[
\deg(\phi_{\mathcal R}) = n(k-1)^{n-1} = 7\cdot 2^6 = 448.
\]
It is easy to verify that ${\mathcal R}$ is not a 3-cylinder; however, its spectrum is invariant under multiplication by any third root of unity (see Lemma 3.11 of \cite{Fan}).  By \cite{Lu} we have
\begin{align*}
    \phi_{\mathcal R} = & x^{m_0}(x^3-1)^{m_1}(x^{15}-13x^{12}+65x^9-147x^6+157x^3-64)^{m_2} \\
    &\cdot (x^6-x^3+2)^{m_3}(x^6-17x^3+64)^{m_4}
\end{align*}
With the intent of applying Theorem \ref{T:3} we need to determine only $c_3, c_6, c_9, c_{12}$.  We have
\begin{align*}
    c_3 &= -240 \\
    c_6 &= 28320 \\
    c_9 &= -2190860 \\
    c_{12} &= 125012034
\end{align*}
By Theorem \ref{T:2} we have $M < 4.5, m > .69, R < 2.25,$ and  $r = 1$ so that at most 252 digits of precision are required to approximate each root.  We compute 
\begin{align*}
    \phi_{\mathcal R} = & x^{133}(x^3-1)^{27}(x^{15}-13x^{12}+65x^9-147x^6+157x^3-64)^{12} \\
    &\cdot (x^6-x^3+2)^6(x^6-17x^3+64)^3
\end{align*}

\begin{figure}[ht]
\begin{center}
\includegraphics[scale = .5]{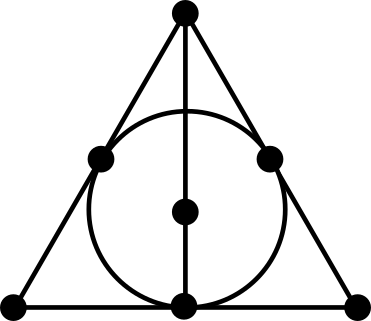} 
\hspace{1cm}
\includegraphics[scale = .5]{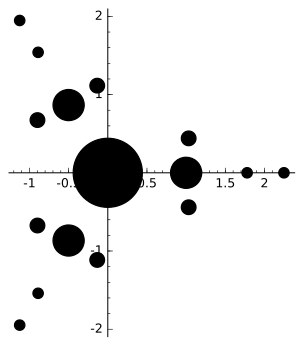} 
\end{center}
\caption{The Rowling hypergraph and its spectrum.}
\label{F:2}
\end{figure}

\section{Acknowledgments}

Thanks to Alexander Duncan for helpful discussions and insights.

\bibliographystyle{amsplain}

\end{document}